\documentclass[12pt,a4paper]{article}

\usepackage{amssymb, latexsym, amsthm}
\usepackage{graphicx} 
\usepackage{url}      
\newcommand{\doi}[1]{\url{http://dx.doi.org/#1}}
\usepackage{amsmath}  

\newcommand{\p}{\partial}
\newcommand{\D}{\Delta}

\renewcommand{\phi}{\varphi}
\newcommand{\e}{\epsilon}

\newcommand{\R}{{\mathbb R}}

\newcommand{\PX}{{\Bbb{P}}}

\newcommand{\spde}{\textsc{spde}}
\newcommand{\pde}{\textsc{pde}}

\newtheorem{theorem}{Theorem}
\newtheorem{lemma}[theorem]{Lemma}

\title{Average and deviation for slow-fast stochastic partial differential equations}

\author{
W. Wang\thanks{School of Mathematical Sciences, University of
Adelaide, Adelaide, \textsc{Australia}. \protect\url{mailto:
w.wang@adelaide.edu.au}; and Department of Mathematics, Nanjing
University, Nanjing, \textsc{China}.
\protect\url{mailto:wangweinju@yahoo.com.cn} }
 \and A.~J. Roberts\thanks{School of Mathematical Sciences, University of Adelaide, Adelaide, \textsc{Australia}.
\protect\url{mailto:anthony.roberts@adelaide.edu.au}}}

\date{\today}

\begin{document}

\maketitle

\begin{abstract}
Averaging is an important method to extract effective macroscopic
dynamics from complex systems with slow modes and fast modes. This
article derives an averaged equation for a class of stochastic
partial differential equations  without any Lipschitz assumption on
the slow modes.  The rate of convergence in probability is obtained
as a byproduct. Importantly, the deviation between the original
equation and the averaged equation is also studied. A martingale
approach proves that the deviation is described by a Gaussian
process. This gives an approximation to errors of~$\mathcal{O}(\e)$
instead of~$\mathcal{O}(\sqrt{\e})$ attained in previous averaging.
\end{abstract}

{\textit {Keywords:}}\hspace{0.2cm} Slow-fast stochastic partial
differential equations, averaging, martingale\vskip1cm

\section{Introduction}\label{sec:intro}
 The need to quantify uncertainties is widely recognized in modeling, analyzing, simulating and predicting complex phenomena~\cite[e.g.]{E00, Imkeller, TM05}. Stochastic partial differential equations (\textsc{spde}s) are appropriate mathematical models for many multiscale systems with uncertain influences~\cite{WaymireDuan}.

Very often a complex system has two widely separated timescales. Then a simplified equation which governs the evolution of the system over the long time scale is highly desirable. Such a simplified equation, capturing the dynamics of the system at the slow time scale, is often called an averaged equation. There is a great deal work on averaging principles for deterministic ordinary differential equations~\cite[e.g.]{Arn, Bog61, Vol62} and for stochastic ordinary differential equations~\cite[e.g.]{Frei98, Kha68, Kifer}. But there are few results on the averaging principle for \spde{}s. Recently, an averaged equation for a system of reaction-diffusion equations with stochastic fast component was obtained by a Lipschitz assumption on all nonlinear terms~\cite{CerFre08}. The resultant averaged equation is deterministic.

This article derives an averaged equation for a class of \spde{}s with stochastic fast component and proves a square-root rate of convergence in probability. Furthermore, the deviation between the original system and the averaged system is determined.

Let $D$ be an open bounded  interval and $L^2(D)$ be the Lebesgue space of square integrable real valued functions on~$D$. Consider the following slow-fast system
 \begin{eqnarray}
 du^\e&=&\big[\D u^\e+f(u^\e, v^\e)\big]\,dt+\sigma_1\,dW_1(t)\,,\quad
   u^\e(0)=u_0\in L^2(D) \,, \label{e:ue}\\
 dv^\e&=&\frac{1}{\e}\big[\D v^\e+g(u^\e, v^\e)\big]\,dt+
   \frac{\sigma_2}{\sqrt{\epsilon}}\,dW_2(t)\,,
    \quad  v^\e(0)=v_0\in L^2(D) \,, \label{e:ve}
 \end{eqnarray}
 with Dirichlet boundary condition. Here $W_1$~and~$W_2$ are mutually independent $L^2(D)$~valued Wiener processes defined on a complete probability space~$(\Omega, \mathcal{F}, \PX)$ detailed in the following section. If for any fixed~$u$, the fast system~(\ref{e:ve}) has unique invariant measure~$\mu_u$, then as $\e\rightarrow 0$\,, under some conditions, the solution~$u^\e$ of~(\ref{e:ue}), converges in probability to the solution of
 \begin{eqnarray}
 du&=&\big[\D u+\bar{f}(u)\big]\,dt+\sigma_1\,dW_1(t) \,, \label{e:au1} \\
  u(0)&=&u_0\quad \text{and} \quad u|_{\p D}=0 \,. \label{e:au2}
 \end{eqnarray}
Here, the average
\begin{equation}\label{e: fbar}
 \bar{f}(u)=\int_Hf(u, v)\mu_u(dv) \,.
 \end{equation}
And the convergence rate is proved to be~$1/2$ in the following sense for any $\kappa>0$
\begin{equation}\label{e:rate}
\mathbb{P}\Big\{\sup_{0\leq t\leq T}|u^\e(t)-u(t)|\leq
C^\kappa_T\e^{1/2}\Big\}>1-\kappa
\end{equation}
for some positive constant~$C^\kappa_T$; see Section~\ref{sec:average}.

 Furthermore by estimate~(\ref{e:rate}), as $\e\rightarrow 0$ the limit of $ (u^\e(t)-u(t))/\sqrt{\e}$ is proved to be a Gaussian process; see Section~\ref{sec:deviation}.

We stress that Theorem~\ref{thm:deviation} gives a much better approximation than the averaged equation.

\section{Preliminaries and main results}\label{sec:Pre}
Let $H=L^2(D)$ with $L^2$-norm denoted by~$|\cdot|$ and inner product by $\langle \cdot, \cdot\rangle$. Define the abstract operator $A=\D$ with zero Dirichlet boundary condition which defines a compact analytic semigroup~$e^{At}$, $t\geq 0$ on~$H$. For any $\alpha>0$\,, $u\in H$~define~$|u|_\alpha=|A^{\alpha/2}u|$ and for $\alpha=1$\,, the norm is denoted as~$\|\cdot\|$. Then let $H_0^\alpha$ be the space the closure of $C_0^\infty(D)$, the space of smooth functions with compact support on~$D$, under the norm $|\cdot|_\alpha$\,. Furthermore, let $H^{-\alpha}$~denote the dual space of~$H_0^\alpha$ and denote by~$\lambda_1$ the first eigenvalue of~$A$\,.  Also we are given~$H$ valued Wiener processes $W_1(t)$~and~$W_2(t)$, $t\geq 0$\,, which are mutually independent on the complete probability space~$(\Omega, \mathcal{F}, \mathcal{F}_t, \PX)$~\cite{PZ92}. Denote by~$\mathbb{E}$ the expectation operator with respect to~$\PX$. Then consider the following \spde{}s with separated time scale
 \begin{eqnarray}
 du^\e&=&\big[Au^\e+f(u^\e, v^\e)\big]\,dt+\sigma_1\,dW_1 \,, \quad u^\e(0)=u_0\in H \,, \label{e:aue}\\
 dv^\e&=&\frac{1}{\e}\big[Av^\e+g(u^\e, v^\e)\big]\,dt+
    \frac{\sigma_2}{\sqrt{\e}}\,dW_2 \,, \quad v^\e(0)=v_0\in H \,. \label{e:ave}
 \end{eqnarray}
 Here $\sigma_1\in\R\,,\sigma_2\neq 0$ are arbitrary real numbers. For our purpose we adopt the following four hypotheses.
 \begin{description}
  \item[H$_1$] $f(x, y):\R\times \R\rightarrow\R$ is continuous, there is a positive constant~$C_f$, such that $f'_x(x, y)\leq C_f$\,, $|f'_y(x,y)|\leq C_f$ and that for any~$x,y$
 \begin{eqnarray*}
 &&|f(x, y)|^2\leq ax^6+by^2+c \,,
 \\&&  f(x, y)x\leq -a x^2-bxy+c \,,
\\&& (f(x_1,y)-f(x_2,y))(x_1-x_2)\leq a(x_1-x_2)^2+c \,,
 \end{eqnarray*}
  for some positive constants~$a$, $b$ and~$c$.
 \item[H$_2$] $g(x, y):\R\times\R\rightarrow\R$ is continuous and is Lipschitz with respect to the both variables with Lipschitz constant~$C_g$. For any~$x,y$
 \begin{eqnarray*}
 g(x, y)y\leq -d y^2+exy
 \end{eqnarray*}
 for some positive constants~$d$ and~$e$.
 \item[H$_3$] $b\geq e$ and $C_g<\lambda_1$\,.
\item[H$_4$] $W_1$ and~$W_2$ are Q-Wiener processes with covariance operator~$Q_1$ and~$Q_2$ respectively. Moreover, $ \operatorname{tr}[A^{1/2}Q_i]<\infty$\,, $i=1,2$\,.
 \end{description}

With the above assumptions we have the first result on the fast component that is proved at the end of Section~\ref{sec:estimate}.
 \begin{theorem} \label{thm:mixing}
 Assume \textbf{H$_1$}--\textbf{H$_4$}. For any fixed $u\in H$\,, system~(\ref{e:ve}) has a unique stationary solution,~$\eta_u^\e(t)$, with distribution~$\mu_u$ independent of~$\e$. Moreover, the stationary measure~$\mu_u$ is exponentially mixing.
 \end{theorem}

Then we prove the following averaging result. \begin{theorem} \label{thm:averaging}
 Assume \textbf{H$_1$}--\textbf{H$_4$}. Given $T>0$\,, for any $u_0\in H$\,, solution~$u^\e(t, u_0)$ of~(\ref{e:ue}) converges in probability to~$u$ in~$C(0, T; H)$ which solves~(\ref{e:au1})--(\ref{e:au2}). Moreover the convergence rate is~$1/2$ that is for any $\kappa>0$
\begin{equation*}
 \mathbb{P}\Big\{\sup_{0\leq t\leq T} |u^\e(t)-u(t)|\leq
 C^\kappa_T\sqrt{\e}\Big\}>1-\kappa
\end{equation*}
for some positive constant $C^\kappa_T>0$\,.
 \end{theorem}

Having the above averaging result we consider the deviation between~$u^\e$ and~$u$. For this we introduce
\begin{equation}\label{e:ze}
 z^\e(t)=\frac{1}{\sqrt{\e}}(u^\e-u) \,.
\end{equation}
Then we  have
 \begin{theorem}\label{thm:deviation}
 $z^\e$ converges in distribution to~$z$ in space~$C(0, T; H)$ which solves
 \begin{equation}\label{e:z}
  \dot{z}=Az+\overline{f'_u}(u)z+\sqrt{B(u)}\dot{\overline{W}}
 \end{equation}
 where $B(u): H\rightarrow H$ is Hilbert--Schmidt with
 \begin{align*}&
 B(u)=2\int_0^\infty\mathbb{E}\Big[(f(u, \eta_u(t))-\bar{f}(u))\otimes(f(u,
 \eta_u(0))-\bar{f}(u))\Big]dt
 \\&
\overline{f'_u}(u)=\int_Hf'_u(u, v)\mu^u(dv)
\end{align*}
and $\overline{W}(t)$ is an $H$-valued cylindrical Wiener process with covariance operator~$\text{Id}_H$.
 \end{theorem}

\section{Some  a priori estimates}\label{sec:estimate}

This section gives some a priori estimates for the solution
of~(\ref{e:aue})--(\ref{e:ave}) which yields the tightness of~$u^\e$
in space~$C(0, T; H)$. First we give a wellposedness result.
\begin{theorem}\label{thm:existence}
 Assume \textbf{H$_1$}--\textbf{H$_4$}. For any $u_0\in H$\,, $v_0\in H$ and any $T>0$\,, there is unique
solution~$(u^\e(t), v^\e(t))$ in $L^2(\Omega, C(0, T; H)\cap L^2(0,
T; H^1_0))$ for~(\ref{e:aue})--(\ref{e:ave})\,.
\end{theorem}
The above result is derived by a standard method~\cite{PZ92} and so is here omitted. Then we have the following estimates for~$(u^\e,
v^\e)$.
\begin{theorem} \label{thm:estimate}
Assume \textbf{H$_1$}--\textbf{H$_4$}. For $u_0\in H^1_0$ and $v_0\in
H_0^1$\,, for any $T>0$\,, there is a positive constant~$C_T$ which is
independent of~$\e$, such that
 \begin{equation}\label{e:Euv}
 \mathbb{E}\sup_{0\leq t\leq T}\|u^\e(t)\|^2+\sup_{0\leq t\leq T}
 \mathbb{E}|v^\e(t)|^2\leq C_T(\|u_0\|^2+|v_0|^2)
 \end{equation}
and for any positive integer~$m$,
\begin{equation}\label{e:Eintuv}
 \mathbb{E}\int_0^T\|u^\e(s)\|^{2m}ds+\mathbb{E}\int_0^T\|v^\e(t)\|^2ds\leq
 C_T(|u_0|^2+|v_0|^2) \,.
 \end{equation}
Moreover
\begin{equation}\label{e:ut}
\mathbb{E}|\dot{u}^\e|_{L^2(0, T; H^{-1})}\leq C_T(|u_0|^2+|v_0|^2) \,.
\end{equation}
\end{theorem}

\begin{proof}
Applying It\^o formula to~$|u^\e(t)|^2$ and~$|v^\e(t)|^2$
respectively and by Gronwall lemma there is positive constant~$C$
which is independent of~$\e$ such that for any $t>0$
\begin{equation}\label{e:4}
 \mathbb{E}|u^\e(t)|^2+ \e\mathbb{E}|v^\e(t)|^2
 \leq C(|u_0|^2+|v_0|^2) \,.
\end{equation}
At the same time we have for any $T>0$ that there is positive
constant~$C_T$ which is independent of~$\e$ such that for $0\leq
t\leq T$
 \begin{equation}\label{e:5}
 \mathbb{E}\int_0^t\|u^\e(s)\|^2ds+\mathbb{E}\int_0^t\|v^\e(s)\|^2ds\leq
 C_T(|u_0|^2+|v_0^2|)
\end{equation}
and
 \begin{equation}\label{e:6}
 \sup_{t\geq 0}\mathbb{E}|v^\e(t)|^2\leq C(|u_0|^2+|v_0|^2) \,.
 \end{equation}

By applying  It\^o formula to~$\|u^\e(t)\|^2$ and using~(\ref{e:5}), a lemma of Da Prato \& Zabczyk's~\cite[Lemma~7.2]{PZ92} proves
\begin{equation}\label{e:supu}
\mathbb{E}\sup_{0\leq t\leq T}\|u^\e(t)\|^2\leq
C_T(\|u_0\|^2+|v_0|^2) \,.
\end{equation}

Further, applying  It\^o formula to~$\|u^\e(t)\|^{2m}$
and~$\|v^\e(t)\|^{2m}$, by the assumption on~$f$ and~(\ref{e:5}), we
then establish the inequality~(\ref{e:Eintuv}). The
bound~(\ref{e:ut}) can be proved by estimating~(\ref{e:Eintuv}) and
the embedding $H^1_0\subset L^6(D)$.
\end{proof}

 Now we show that $\mathcal{L}(u^\e)$, the distribution of~$u^\e$, is tight in~$C(0, T; H)$. For this we need the following lemma by Simon~\cite{Simon}.
\begin{lemma}\label{embedding}
Let~$E$, $E_0$ and~$E_1$ be Banach spaces such that $E_1\Subset
E_0$\,, the interpolation space $(E_0, E_1)_{\theta,1}\subset E$ with
$\theta\in (0, 1)$  and $E\subset E_0$ with $\subset$ and~$\Subset$
denoting continuous and compact embedding respectively. Suppose
$p_0,p_1\in [1,\infty]$ and $T>0$\,, such that
\begin{displaymath}
\mathcal{V} \text{ is a bounded set in } L^{p_1}(0, T; E_1)
\end{displaymath}
and
\begin{displaymath}
\p\mathcal{V}:=\{\p v: v\in \mathcal{V}\} \text{ is
a bounded set in } L^{p_0}(0, T; E_0).
\end{displaymath}
Here $\p$ denotes the distributional derivative. If
$1-\theta>1/p_\theta$ with
 \begin{displaymath}
\frac{1}{p_\theta}=\frac{1-\theta}{p_0}+\frac{\theta}{p_1}\,,
 \end{displaymath}
then $\mathcal{V}$ is relatively compact in~$C(0, T; E)$.
\end{lemma}

By the above lemma we have the following result.
 \begin{theorem}\label{thm:tight}
Assume \textbf{H$_1$}--\textbf{H$_4$}. $\{\mathcal{L}(u^\e)\}_{\e>0}$ is
tight in space~$C(0, T; H)$.
 \end{theorem}

 \begin{proof}
Taking $E_0=H^{-1}$, $E=H$ and $E_1=H_0^1$ and $p_0=2$\,,
$\theta=1/2$\,. By Theorem~\ref{thm:estimate}, $p_1$~can be taken as
arbitrary lager positive integer, then by Lemma~\ref{embedding},
$\{\mathcal{L}(u^\e)\}_{\e>0}$ is tight in space~$C(0, T; H)$.
 \end{proof}

\begin{proof}[Proof of Theorem~\ref{thm:mixing}] For any two solutions~$v_1^\e$ and~$v_2^\e$, the It\^o formula yields
\begin{equation*}
\mathbb{E}|v_1^\e(t)-v_2^\e(t)|^2\leq
e^{-2(\lambda_1-C_g)t/\e}\mathbb{E}|v_1(0)-v_2(0)| \,,
\end{equation*}
which means the existence of a unique stationary solution~$\eta_u^\e$ for~(\ref{e:ave}) distributes as~$\mu_u$ such that for
any $v_0\in H$
\begin{equation}\label{e:stat}
\mathbb{E}|v^\e(t)-\eta^\e_u(t)|^2\leq
e^{-2(\lambda_1-C_g)t/\e}\mathbb{E}|v_0-\eta_u^\e(0)|^2 \,,
\end{equation}
which yields the exponential mixing. Moreover, since $|f'_y(x,
y)|\leq C_f$\,, we also have
 \begin{equation}\label{e:expmix2}
\mathbb{E}\Big |f(u, v^\e(t))-\int_Hf(u, x)\mu_u(dx)\Big|\leq
C(1+|v_0|^2)e^{-2(\lambda_1-C_g) t/\e } \,.
 \end{equation}

By the time scale transformation $t\rightarrow \tau=\e t$\,, (\ref{e:ave})
is transformed to
\begin{equation}\label{e:vslow}
 dv=\big[Av+g(u, v)\big]\,d\tau+
 \sigma_2\,d\tilde{W}_2(\tau)\,, \quad  v(0)=v_0 \,,
\end{equation}
where $\tilde{W}_2$ is the scaled version of~$W_2$ and with the same
distribution. Then the \spde~(\ref{e:vslow}) has a unique stationary solution~$\eta_u$ with distribution~$\mu_u$. And by the ergodic property of~$\mu_u$, we have
\begin{equation}\label{e:barf}
\bar{f}(u)=\lim_{t\rightarrow\infty}\frac{1}{t}\int_0^tf(u,
\eta_u(s))\,ds \,.
\end{equation}
Furthermore, by a generalized theorem on contracting maps depending on
a parameter~\cite[Appendix~C]{Cer01},~\cite{CerFre08},
$v_u(t)$~is differential with respect~$u$ with
\begin{equation}\label{e:Dv}
\sup_{u, v_0\in H,\ 0\leq t< \infty}|D_uv_u|_{\mathcal{L}(H)}\leq C
\end{equation}
for some positive constant $C$\,.
\end{proof}

We end this section by giving the following a priori estimates on
the solutions of the averaged
equation~(\ref{e:au1})--(\ref{e:au2})\,. First we need a local
Lipschitz property of~$\bar{f}$ which is yielded by~(\ref{e:Dv}) and
the following estimate, for any $u_1$\,, $u_1\in H_0^1$
\begin{eqnarray}
&& \frac{1}{\tau}\Big|\int_0^\tau [f(u_1, v_{u_1}(s))-f(u_2,
 v_{u_2}(s))\big]ds\Big|_H\nonumber\\
 &\leq & \frac{1}{\tau}\Big|\int_0^\tau[f(u_1, v_{u_1}(s))-f(u_2,
 v_{u_1}(s))\big]ds\Big|_H\nonumber
 \\ &&{}
 +\frac{1}{\tau}\Big|\int_0^\tau [f(u_2, v_{u_1}(s))-f(u_2,
 v_{u_2}(s))\big]ds\Big|_H\nonumber\\
 &\leq & 2\Big[\|u_1\|^2+\|u_2\|^2+
 \sup_{u, v_0\in H, \ 0\leq t<\infty}|D_uv(t)|_{\mathcal{L}(H)}\Big]|u_1-u_2| \,.\label{e:barf-Lip}
\end{eqnarray}

\begin{lemma}\label{lem:est-avu}
Assume $\textbf{H}_1$--$\textbf{H}_4$\,. For any $u_0\in H_0^1$, for
any $T>0$, (\ref{e:au1})--(\ref{e:au2}) has a unique solution $u\in
L^2(\Omega, C(0, T; H)\cap L^2(0, T; H_0^1))$\,.
 Moreover there is a positive constant $C_T$ such that for any positive
integer $m\geq 2$ and any $0\leq t\leq T$
\begin{equation}\label{e:u}
\mathbb{E}\|u(t)\|^m\leq C_T(1+\|u_0\|^m)\,.
\end{equation}
\end{lemma}
\begin{proof}
Applying It\^o formula to $|u(t)|^2$ yields
\begin{equation*}
\frac{1}{2}\frac{d}{dt}|u(t)|^2=-\|u(t)\|^2+\langle \bar{f}(u(t)),
u(t)\rangle+ \sigma_1\langle u(t),
\dot{W}_1\rangle+\frac{\sigma_1^2}{2}\operatorname{tr}Q_1\,.
\end{equation*}
By~(\ref{e:barf}) and assumption $\textbf{H}_1$
\begin{eqnarray*}
\langle \bar{f}(u(t)), u(t)\rangle&=&\lim_{s\rightarrow \infty}
\frac{1}{s}\int_0^s\big\langle
 f(u(t),\eta_{u(t)}(\tau) ), u(t)\big\rangle\,d\tau\\
 &\leq & -a|u(t)|^2-b\langle \bar{\eta}_{u(t)}, u(t)\rangle+c\,.
\end{eqnarray*}
Then by the Gronwall lemma and (\ref{e:6}) there is some positive
constant~$C$ such that for any $t>0$
\begin{equation*}
\mathbb{E}|u(t)|^2\leq C(1+|u_0|^2)\,.
\end{equation*}
Now by the same analysis of the proof of Theorem \ref{thm:estimate}
we obtain (\ref{e:u})\,. Then by the above a priori estimates and
the local Lipschitz property of~$\bar{f}$, a standard
method~\cite{PZ92} yields the existence and uniqueness of~$u$.
\end{proof}

\section{Averaged equation} \label{sec:average}
This section gives the averaged equation and, as a byproduct, the
convergence rate is obtained. For this we consider our system
in a smaller probability space. By Theorem~\ref{thm:tight}, for any
$\kappa>0$ there is compact set~$K_\kappa$ in~$C(0, T; H)$ such that
\begin{displaymath}
\PX\{u^\e\in K_\kappa \}> 1-\kappa \,.
\end{displaymath}
Here $K_\kappa$ is chosen as a family of decreasing sets with
respect to~$\kappa$. Moreover, by the estimate~(\ref{e:Euv}) and
Markov inequality, we further choose the set~$K_\kappa$ such that
for $u^\e\in K_\kappa$
\begin{displaymath}
\|u^\e(t)\|^2\leq C_T^\kappa \,, \quad  t\in [0, T] \,,
\end{displaymath}
for some positive constant~$C_T^\kappa$.

\begin{proof}[Proof of Theorem~\ref{thm:averaging}]
Now we prove the rate of convergence. In order to do this, for any
$\kappa>0$ we introduce a new sub-probability space~$(\Omega_\kappa,
\mathcal{F}_\kappa, \PX_\kappa)$ defined by
\begin{equation*}
 \Omega_\kappa=\{\omega\in \Omega: u^\e(\omega)\in
K_\kappa\} \,, \quad \mathcal{F}_\kappa=\{S\cap \Omega_\kappa:
S\in\mathcal{F}\}
\end{equation*}
and
\begin{displaymath}
\PX_\kappa(S)=\frac{\PX(S\cap\Omega_\kappa)}{\PX(\Omega_\kappa)}
\quad \text{for} \quad  S\in\mathcal{F}_\kappa\,.
\end{displaymath}
Then $\PX(\Omega\setminus\Omega_\kappa)\leq\kappa$\,. In the
following we denote by~$\mathbb{E}_\kappa$ the expectation operator
with respect to~$\PX_\kappa$.

Now we restrict $\omega\in\Omega_\kappa$ and introduce an auxiliary
process. For any $T>0$\,, partition the interval~$[0, T]$ into
subintervals of length $\delta=\sqrt{\e}$\,. Then we construct
processes~$(\tilde{u}^\e, \tilde{v}^\e)$ such that for $t\in
[k\delta, (k+1)\delta)$,
 \begin{eqnarray}
\tilde{u}^\e(t)&=&e^{A(t-k\delta)}u^\e(k\delta)+
\int_{k\delta}^te^{A(t-s)}f(u^\e(k\delta),\tilde{v}^\e(s))\,ds
   \nonumber \\&& {} +\sigma_1\int_{k\delta}^te^{A(t-s)}dW_1(s) \,, \quad   \tilde{u}^\e(0)=u_0\label{e:tue}\\
d\tilde{v}^\e(t)&=&\frac{1}{\e}\big[B\tilde{v}^\e(t)+g(u^\e(k\delta),
\tilde{v}^\e(t))\big]\,dt+\frac{\sigma_2}{\sqrt{\epsilon}}\,dW_2(t)
\,, \nonumber\\&&  \tilde{v}^\e(k\delta)=v^\e(k\delta)\label{e:tve}
\,.
 \end{eqnarray}
By the It\^o formula for $t\in [k\delta, (k+1)\delta)$
\begin{eqnarray*}
&&\frac{1}{2}\frac{d}{dt}|v^\e(t)-\tilde{v}^\e(t)|^2\\ &\leq&
-\frac{1}{\e}(\lambda_1-L_g)|v^\e(t)-\tilde{v}^\e(t)|^2+
\frac{1}{\e}L_g|v^\e(t)-\tilde{v}^\e(t)||u^\e(t)-u^\e(k\delta)| \,.
\end{eqnarray*}
By the choice of~$\Omega_\kappa$, $K_\kappa$ is compact in space
$C(0,T;H)$, there is $C^\kappa_T>0$\,, such that
\begin{equation}\label{e:u1}
|u^\e(t)-u^\e(k\delta)|^2\leq C^\kappa_T\delta^2
\end{equation}
 for $t\in [k\delta,
(k+1)\delta)$. Then by the Gronwall lemma,
\begin{equation}\label{e:v1}
 |v^\e(t)-\tilde{v}^\e(t)|^2\leq C_T\delta^2 \,, \quad  t\in[0, T] \,.
\end{equation}
Moreover, by the choice of~$\Omega_\kappa$ and the assumption on the
growth of~$f(\cdot, v)$, $f(\cdot, v): H\rightarrow H^{-\beta} $ is
Lipschitz with $-1/2\leq\beta\leq -1/4$\,. Then we have for $t\in
[k\delta, (k+1)\delta)$
\begin{eqnarray*}
|u^\e(t)-\tilde{u}^\e(t)|&\leq&
 C_f\int_{k\delta}^t|v^\e(s)-\tilde{v}^\e(s)|\,ds+
C_fC^\kappa_\beta\int_{k\delta}^t|u^\e(k\delta)-u^\e(s)|\,ds \,,
\end{eqnarray*}
for some positive constant $C_\beta^\kappa$\,. So by
noticing~(\ref{e:u1}), we have
\begin{equation}\label{e:u2}
|u^\e(t)-\tilde{u}^\e(t)|\leq C^\kappa_T\delta \,, \quad  t\in[0, T]
\,.
\end{equation}
On the other hand, in the mild sense
\begin{equation*}
u(t)=e^{At}u_0+\int_0^te^{A(t-s)}\bar{f}(u(s))\,ds+
\sigma_1\int_0^te^{A(t-s)}\,dW_1(s) \,.
\end{equation*}
Then, using $\lfloor z\rfloor$~to denote the largest integer less than
or equal to~$z$,
\begin{eqnarray*}
| \tilde{u}^\e(t)-u(t)|&\leq&\int_0^te^{A(t-s)}\big|f(u^\e(\lfloor
s/\delta\rfloor\delta), \tilde{v}^\e(s))-\bar{f}(u^\e(\lfloor
s/\delta\rfloor\delta))\big|\,ds \\
&&{}+\int_0^te^{A(t-s)}\big|\bar{f}(u^\e(\lfloor
s/\delta\rfloor\delta)
)-\bar{f}(u^\e(s))\big|\,ds\\
&&{}+\int_0^te^{A(t-s)}\big|\bar{f}(u^\e(s))-\bar{f}(u(s))\big|\,ds
\,.
\end{eqnarray*}
Then by~(\ref{e:expmix2}), (\ref{e:barf}) and~(\ref{e:barf-Lip}) we
have for $t\in [0, T]$
\begin{equation}\label{e:u3}
|\tilde{u}^\e(t)-u(t)|\leq
C^\kappa_T\left[\delta+\int_0^T|u^\e(s)-u(s)|\,ds\right] \,.
\end{equation}
As
\begin{equation*}
|u^\e(t)-u(t)|\leq |u^\e(t)-\tilde{u}(t)|+|\tilde{u}(t)-u(t)|
\end{equation*}
by the Gronwall lemma and~(\ref{e:u1}), (\ref{e:u2}) and~(\ref{e:u3}) we have for $t\in [0, T]$,
\begin{equation}\label{e:u0}
|u^\e(t)-u(t)|\leq C^\kappa_T\sqrt{\e} \,.
\end{equation}
The proof of Theorem~\ref{thm:averaging} is complete.
\end{proof}

\section{Deviation estimate}\label{sec:deviation}
The previous section proved that for any $T>0$ in the sense of
probability
\begin{equation*}
    \sup_{t\in [0, T]}|u^\e(t)-u(t)|\leq C_T\sqrt{\e}
\end{equation*}
for some positive constant~$C_T$. Formally we should have the
following form $u^\e(t)=u(t)+\mathcal{O}(\e^{{1}/{2}})$\,. This
section determines the coefficient of $\e^{{1}/{2}}$, the
deviation.\\

\begin{proof}[Proof of Theorem~\ref{thm:deviation}]
We approximate the deviation~$z^\e$ defined by~(\ref{e:ze}) for
small $\e>0$. The deviation~$z^\e$ satisfies
\begin{equation}\label{e:zequ}
 \dot{z}^\e=A z^\e+\frac{1}{\sqrt{\e}}\big[f(u^\e, v^\e)-\bar{f}(u)\big]
\end{equation}
with $z^\e(0)=0$\,. By the assumption on~$f$ we have
\begin{eqnarray}
 \frac{1}{2}\frac{d}{dt}|z^\e|^2
  &\leq &-\|z^\e\|^2+C_f|z^\e|^2+
 \frac{1}{\sqrt{\e}}C_f|v^\e-\eta^\e_u||z^\e|\nonumber\\&&{}+
  \frac{1}{\sqrt{\e}}\langle f(u, \eta^\e_u)-\bar{f}(u),
  z^\e\rangle \,. \label{e:f}
\end{eqnarray}
Then the Gronwall lemma yields that for any $T>0$\,,
\begin{equation}\label{e:Ez}
\mathbb{E}\sup_{0\leq t\leq
T}|z^\e(t)|^2+\mathbb{E}\int_0^T\|z^\e(t)\|^2dt\leq
C_T(|u_0|^2+|v_0|^2) \,.
\end{equation}
In the mild sense we write
 \begin{displaymath}
z^\e(t)=\frac{1}{\sqrt{\e}}\int_0^te^{A(t-r)}[f(u^\e(r),
v^\e(r))-\bar{f}(u(r))]\,dr.
 \end{displaymath}
Then for any $0\leq s<t$\,, by the property of~$e^{At}$, we have for
some positive $1>\delta>0$
\begin{eqnarray*}
|z^\e(t)-z^\e(s)|&\leq& \frac{1}{\sqrt{\e}}\left|\int_0^te^{A(t-r)}
f(u^\e(r), v^\e(r))-\bar{f}(u(r))\,dr
\right.\\ &&\left.{}-\int_0^se^{A(s-r)}
f(u^\e(r), v^\e(r))-\bar{f}(u(r) )\,dr\right|\\
&\leq& C_T|t-s|^\delta\frac{1}{\sqrt{\e}}\big|f(u^\e,
v^\e)-\bar{f}(u)\big|_{L^2(0, T; H)} \,.
\end{eqnarray*}
By the assumption~$\mathbf{H}_1$, Theorem~\ref{thm:estimate} and~(\ref{e:Ez})
\begin{displaymath}
\mathbb{E}\frac{1}{\sqrt{\e}}\big|f(u^\e,
v^\e)-\bar{f}(u)\big|_{L^2(0, T; H)}\leq C_T(|u_0|^2+|v_0^2|)\,.
\end{displaymath}
Then we have
\begin{equation}\label{e:zCtheta}
\mathbb{E}|z^\e(t)|_{C^\delta(0, T; H)}\leq C_T(|u_0|^2+|v_0|^2) \,.
\end{equation}
Here $C^\delta(0, T; H)$ is the H\"older space  with exponent~$\delta$.
 On the other hand, also by the property of~$e^{At}$, we have
for some positive constant~$C_{T, \alpha}$ and for some $1>\alpha>0$
\begin{eqnarray*}
&&|z^\e(t)|_{H^\alpha}\leq\frac{1}{\sqrt{\e}}\int_0^t(t-s)^{-\alpha/2}
\big|f(u^\e(s), v^\e(s))-\bar{f}(u(s))\big|_H\,ds \nonumber\\
&\leq& C_{T, \alpha}\frac{1}{\sqrt{\e}}\big|f(u^\e,
v^\e)-\bar{f}(u)\big|_{L^2(0, T; H)} \,.
\end{eqnarray*}
Then
\begin{equation}\label{e:zCtheta2}
\mathbb{E}\sup_{0\leq t\leq T}|z^\e(t)|_{H^\alpha}\leq C_{T,
\alpha}(|u_0|^2+|v_0^2|) \,.
\end{equation}
And by the compact embedding of $C^\delta(0, T; H)\cap C(0, T;
H^\alpha)\subset C(0, T; H)$, $\{\nu^\e\}_\e$\,, the distribution
of~$\{z^\e\}_\e$\,, is tight in~$C(0, T; H)$.

 Divide~$z^\e$ into $z_1^\e+z_2^\e$ which solves
\begin{equation}\label{e:z1}
\dot{z}_1^\e=A z_1^\e+\frac{1}{\sqrt{\e}}[f(u,
\eta_u^\e)-\bar{f}(u)] \,, \quad  z_1(0)^\e=0
\end{equation}
and
\begin{equation}\label{e:z2}
\dot{z}_2^\e=A z_2^\e+\frac{1}{\sqrt{\e}}[f(u^\e, v^\e)-f(u,
\eta_u^\e)] \,, \quad z_2^\e(0)=0
\end{equation}
respectively and consider~$z_1^\e$ and~$z_2^\e$ separately. We
follow a martingale approach~\cite{Kes79, Wata88}. Denote
by~$\nu_1^\e$ be the probability measure of~$z_1^\e$ induced on
space~$C(0,T; H)$.  For $\gamma>0$\,, denote by~$UC^\gamma(H, \R)$
the space of all functions from~$H$ to~$\R$ which, together with all
Fr\'echet derivatives to order~$\gamma$, are
uniformly continuous. For $h\in UC^\gamma(H, \R)$\,, denote by $h'$
and $h''$ the first and second order Fr\'echet derivative. Then we
have the following lemma.
\begin{lemma}\label{lem:mart}
Assume $\mathbf{H}_1$--$\mathbf{H}_4$. Any limiting measure of~$\nu_1^\e$, denote by~$P^0$, solves the following martingale problem
on~$C(0, T; H)$: $P^0\{z_1(0)=0\}=1$\,,
\begin{displaymath}
h(z_1(t))-h(z_1(0))-\int_0^t \langle h'(z_1(\tau)),
Az_1(\tau)\rangle\,d\tau-\frac{1}{2}\int_0^t
 \operatorname{tr}\big[h''(z_1(\tau))B(u)\big ]\,d\tau
\end{displaymath}
is a $P^0$ martingale for any $h\in UC^2(H, \R)$. Here
\begin{displaymath}
B(u) = 2\int_0^\infty \mathbb{E} \big[(f(u,
\eta_u(t))-\bar{f}(u))\otimes (f(u, \eta_u(0))-\bar{f}(u))\big]\,dt
\end{displaymath}
and~$\otimes$ denotes the tensor product.
\end{lemma}

\begin{proof}
We follow a martingale approach~\cite{Kes79, Wata88}\,. For any
$0<s\leq t<\infty$ and $h\in UC^\infty(H)$ we have
\begin{eqnarray*}
&&h(z_1^\e(t))-h(z_1^\e(s))= \int_s^t\langle
h'(z_1^\e(\tau)),\frac{dz_1^\e}{dt}\rangle\,d\tau\\&=&
\int_s^t\langle h'(z_1^\e(\tau)), Az_1^\e(\tau)\rangle\,d\tau+
\frac{1}{\sqrt{\e}}\int_s^t\langle h'(z_1^\e(\tau)), f(u(\tau),
\eta^\e_u(\tau))-\bar{f}(u(\tau)) \rangle\,d\tau \,.
\end{eqnarray*}
Rewrite the second term as
 \begin{eqnarray*}
&&\frac{1}{\sqrt{\e}}\int_s^t\langle h'(z_1^\e(\tau)), f(u(\tau),
\eta^\e_u(\tau))-\bar{f}(u(\tau))\rangle \, d\tau \\
&=& \frac{1}{\sqrt{\e}}\int_s^t\langle h'(z_1^\e(t)), f(u(\tau),
\eta^\e_u(\tau))-\bar{f}(u(\tau)) \rangle\, d\tau\\
&&{}-\frac{1}{\sqrt{\e}}\int_s^t\int_\tau^t  h''(z_1^\e(\delta))
\Big[f(u(\tau), \eta^\e_u(\tau))-\bar{f}(u(\tau))), Az_1^\e(\delta)\Big]\, d\delta \,d\tau\\
&&{}-\frac{1}{\e}\int_s^t\int_\tau^t
h''(z_1^\e(\delta))\Big[(f(u(\tau),
\eta^\e_u(\tau))-\bar{f}(u(\tau))), f(u(\delta),
\eta^\e_u(\delta))\\&&\quad{}-\bar{f}(u(\delta)) \Big]\, d\delta\,
d\tau\\&=&L_1+L_2+L_3
\end{eqnarray*}
where $L_1$, $L_2$ and~$L_3$ denote the separate lines of the
right-hand side of the above equation, respectively. Let
$\{e_i\}_{i=1}^\infty$ be one eigenbasis of~$H$, then
\begin{eqnarray*}
&&h''(z_1^\e(\delta))\Big((f(u(\tau),
\eta^\e_u(\tau))-\bar{f}(u(\tau))), f(u(\delta),
\eta^\e_u(\delta))-\bar{f}(u(\delta))
\Big)\\&=&\sum^\infty_{i,j=1}\p_{ij}h(z^\e_1(\delta))\big\langle
(f(u(\tau), \eta^\e_u(\tau))-\bar{f}(u(\tau)))
\\ && \hspace{6cm}\otimes(f(u(\delta), \eta^\e_u(\delta))-\bar{f}(u(\delta))),  e_i\otimes e_j
\big\rangle \,.
\end{eqnarray*}
Here $\p_{ij}=\p_{e_i}\p_{e_j}$ where $\p_{e_i}$~is the directional
derivative in direction~$e_i$ and $\otimes$  denotes the tensor
product.

Denote by $A_{ij}^\e(\delta, \tau)=\big\langle(f(u(\tau),
\eta^\e_{u(\tau)}(\tau))-\bar{f}(u(\tau))) \otimes(f(u(\delta),
\eta^\e_{u(\delta)}(\delta))-\bar{f}(u(\delta))), e_i\otimes
e_j\big\rangle$\,. Then we have
\begin{eqnarray*}
 L_3\hspace{-0.2cm}&=&\hspace{-0.2cm}-\frac{1}{\e}\sum_{ij}\int_s^t\int_\tau^t
 \p_{ij}h(z_1^\e(\delta))A_{ij}^\e(\delta,\tau)\, d\delta\, d\tau\\
\hspace{-0.2cm}&=&\hspace{-0.2cm}-\frac{1}{\e}\sum_{ij}\int_s^t\int_\tau^t
\int_\delta^t\big\langle \p_{ij}h'(z_1^\e(\lambda)),
\big[Az_1^\e(\lambda)+\frac{1}{\sqrt{\e}}(f(u(\lambda),
\eta_{u(\lambda)}^\e(\lambda))-\bar{f}(u(\lambda)))\big]\big\rangle
\\&&\qquad{}\times \tilde{A}_{ij}^\e(\delta,\tau)\,
 d\lambda\, d\delta\, d\tau
\\&&{}
 +\frac{1}{\e}\sum_{ij}\int_s^t\int_\tau^t\p_{ij}h(z_1^\e(t))\tilde{A}_{ij}^\e
(\delta,\tau)\, d\delta\,
d\tau\\&&{}+\frac{1}{\e}\sum_{ij}\int_s^t\int_s^\tau\p_{ij}h(z_1^\e(\tau))
\mathbb{E}[A_{ij}^\e
(\delta,\tau)]\, d\delta \, d\tau\\
\hspace{-0.2cm}&=&L_{31}+L_{32}+L_{33}
\end{eqnarray*}
with $\tilde{A}_{ij}^\e(\delta, \tau)=A_{ij}^\e(\delta,
\tau)-\mathbb{E}[A_{ij}^\e(\delta, \tau)]$. For our purpose, for any
bounded continuous function~$\Phi$ on~$C(0,s; H)$, let $\Phi(\cdot,
\omega)=\Phi(z_1^\e(\cdot, \omega))$.
  Then by~(\ref{e:expmix2}), we have the following estimate
\begin{eqnarray*}
|\mathbb{E}[(L_{31}+L_{32})\Phi]|\rightarrow 0  \text{ as } \e\rightarrow 0 \,.
\end{eqnarray*}
Now we determine the limit of
$\int_s^\tau\mathbb{E}A^\e_{ij}(\delta, \tau)\,d\delta$ as
$\e\rightarrow 0$\,. Notice that $\eta_{u}$ depends on $u$\,,
$A^\e_{ij}(\delta, \tau)$ is not a stationary process
 for fixed $\tau$\,. For this introduce
\begin{equation*}
\overline{A}_{ij}^\e(\delta, \tau)=\left\langle(f(u(\tau),
\eta^\e_{u(\tau)}(\tau))-\bar{f}(u(\tau))) \otimes(f(u(\tau),
\eta^\e_{u(\tau)}(\delta))-\bar{f}(u(\tau))), e_i\otimes
e_j\right\rangle\,.
\end{equation*}
Then
\begin{eqnarray*}
&&\Big|\int_s^\tau \mathbb{E}[A^\e_{ij}(\delta,
\tau)-\overline{A}^\e_{ij}(\delta,
\tau)]\,d\delta\Big|\\&\leq&\int_s^\tau \mathbb{E}\Big|\langle
f(u(\tau), \eta^\e_{u(\tau)}(\tau))-\bar{f}(u(\tau)), e_i\rangle
\big[\langle f(u(\delta),
\eta^\e_{u(\delta)}(\delta))-\bar{f}(u(\delta)), e_j\rangle\\&&{}-
\langle f(u(\tau), \eta^\e_{u(\tau)}(\delta))-\bar{f}(u(\tau)),
e_j\rangle \big]\Big|\,d\delta\\
&=& \int_s^\tau\mathbb{E}\Big|\langle f(u(\tau),
\eta^\e_{u(\tau)}(\tau))-\bar{f}(u(\tau)), e_i\rangle \big[\langle
f(u(\delta), \eta^\e_{u(\delta)}(\delta))-f(u(\delta),
\eta^\e_{u(\tau)}(\delta)), e_j\rangle\\&&{}+\langle f(u(\delta),
\eta^\e_{u(\tau)}(\delta))-f(u(\tau), \eta^\e_{u(\tau)}(\delta)),
e_j\rangle \big]\Big|\,d\delta\,.
\end{eqnarray*}
By the assumption $\textbf{H}_1$ and (\ref{e:Dv}) we have
\begin{eqnarray*}
&&\big|\langle f(u(\delta),
\eta^\e_{u(\delta)}(\delta))-f(u(\delta),
\eta^\e_{u(\tau)}(\delta)), e_j\rangle\big|\\&\leq&
C_f|\eta^\e_{u(\delta)}(\delta)-\eta^\e_{u(\tau)}
(\delta)\big||e_j|\\
&\leq&C_fC|e_j|
\end{eqnarray*}
and by Lemma \ref{lem:est-avu}
\begin{eqnarray*}
&&\big|\langle f(u(\delta), \eta^\e_{u(\tau)}(\delta))-f(u(\tau),
\eta^\e_{u(\tau)}(\delta)), e_j\rangle\big|\\&\leq&
\big[\|u(\delta)\|^2+\|u(\tau)\|^2 \big]|u(\delta)-u(\tau)|
|e_j|\\&\leq&C_T(1+\|u_0\|^3)|e_j| \,.
\end{eqnarray*}
Then by~(\ref{e:expmix2})
\begin{equation*}
\Big|\frac{1}{\e}\int_s^\tau \mathbb{E}[A^\e_{ij}(\delta,
\tau)-\overline{A}^\e_{ij}(\delta, \tau)]\,d\delta\Big|\rightarrow
0\,, \quad \text{as}\; \e\rightarrow 0\,.
\end{equation*}

Now for fixed $\tau$\,, since $\eta_{u(\tau)}(t)$ is stationary
correlated, we put
\begin{eqnarray*}
b^{ij}_{u(\tau)}(\delta-\tau
)&=&\mathbb{E}\big[\big\langle\big(f(u(\tau),
\eta_{u(\tau)}(\delta))-\bar{f}(u(\tau))\big) \\&&{}\otimes\big(
f(u(\tau), \eta_{u(\tau)}(\tau))-\bar{f}(u(\tau))\big), e_i\otimes
e_j \big\rangle \big] \,.
\end{eqnarray*}
Then we have
\begin{equation*}
\mathbb{E}\left[\overline{A}_{ij}^\e(\delta,
\tau)\right]=b^{ij}_{u(\tau)}\Big(\frac{\delta-\tau}{\e}\Big) \,.
\end{equation*}
Further by the exponential mixing property, for any fixed
$\delta>\tau$
\begin{equation*}
\int_0^{(\delta-\tau)/\e}b^{ij}_{u(\tau)}(\lambda) \,
d\lambda\rightarrow \int_0^\infty b^{ij}_{u(\tau)}(\lambda) \,
d\lambda=:\frac{1}{2}B_{ij}(u(\tau))\,, \quad \e\rightarrow 0\,.
\end{equation*}
Then, if $\e_n\rightarrow 0$ as $n\rightarrow\infty$\,,
$\nu^{\e_n}\rightarrow P^0 $\,,
\begin{equation*}
\lim_{n\rightarrow\infty}\mathbb{E}[L_3\Phi] =\frac{1}{2}
\int_s^t\mathbb{E}^{P^0}\Big(
\operatorname{tr}\big[h''(z_1(\tau))B(u(\tau))\big]\Phi\Big) d\tau
\,,
\end{equation*}
where
\begin{equation*}
B(u)=\sum_{i, j}B_{ij}(u)(e_i\otimes e_j)\,.
\end{equation*}
Moreover by the assumption on $f$ and the estimates of Lemma
~\ref{lem:est-avu}, $B(u):~H\rightarrow H$ is Hilbert--Schmidt.

 Similarly by~(\ref{e:expmix2})
\begin{equation*}
\mathbb{E}[L_1\Phi+L_2\Phi]\rightarrow 0 \text{ as }  \e\rightarrow 0 \,.
\end{equation*}
By the tightness of~$z^\e$ in~$C(0, T; H)$, the sequence~$z_1^{\e_n}$ has a limit process, denote by~$z_1$, in weak sense.
Then
\begin{equation*}
\lim_{n\rightarrow\infty}\mathbb{E} \Big[\int_s^t\langle
h'(z_1^{\e_n}(\tau)), Az_1^{\e_n}(\tau)\rangle\Phi
\, d\tau\Big]=\mathbb{E}\Big[\int_s^t \langle h'(z_1(\tau)),
Az_1(\tau)\rangle \Phi \, d\tau \Big]
\end{equation*}
and
\begin{equation*}
\lim_{n\rightarrow\infty}\mathbb{E}\big[\big(h(z_1^{\e_n}(t))-
h(z_1^{\e_n}(s))\big)\Phi\big]=\mathbb{E}\big[\big(h(z_1(t))-h(z_1(s))\big)\Phi\big] \,.
\end{equation*}
At last we have
\begin{eqnarray}\label{e:z1mart}
&&\mathbb{E}^{P^0}\big[\big(h(z_1)(t)-h(z_1(s))\big)\Phi\big]\\
&=&\mathbb{E}^{P^0}\Big[\int_s^t\langle h'(z_1(\tau)),
Az_1(\tau)\rangle \Phi \, d\tau
\Big]\nonumber\\&&{}+\frac{1}{2}\mathbb{E}^{P^0}\left\{\int_s^t
\operatorname{tr} \big[h''(z_1(\tau))B(u(\tau))\big]\Phi \, d\tau
\right\}.\nonumber
\end{eqnarray}
By an approximation argument we can prove~(\ref{e:z1mart}) holds for
all $h\in UC^2(H)$. This completes the proof.
\end{proof}

We need a lemma on the martingale problem. First introduce some
notation. Suppose~$\mathfrak{A}$ is a generator of bounded analytic
compact semigroup~$S(t)$. $\mathfrak{F}: [0, T]\times H\rightarrow
H$ and $\mathfrak{B}: [0, T]\times H\rightarrow \mathcal{L}(H)$ is~$\mathcal{B}(H)$ and $\mathcal{B}(\mathcal{L}_2(H))$ measurable and
bounded. Let $\mathfrak{L}_t$~be a second order Kolmogorov diffusion
operator of the form
\begin{equation*}
\mathfrak{L}_tF=\frac{1}{2} \operatorname{tr}[
\mathfrak{B}Q\mathfrak{B}^*F_{zz}]+\langle \mathfrak{A}z, F_z
\rangle+\langle \mathfrak{F}(z), F_z \rangle
\end{equation*}
for any bounded continuous function~$F$ on~$H$ with first and second
order Fr\'echet derivatives. Then we have the following
result~\cite{MM88}.
\begin{lemma}For any $T>0$\,,
$F(X(t))-F(X(0))-\int_0^t\mathfrak{L}_sF(X(s))\,ds$ is a  $P^0$
  martingale on space $C(0,T; H)$ if and only if the following
  equation
\begin{equation*}
dX(t)=\left[\mathfrak{A}X(t)+\mathfrak{F}(X(t))\right]\,dt+
\mathfrak{B}\,dW(t)
\end{equation*}
has a weak solution $(\tilde{\Omega},
\{\tilde{\mathcal{F}}_t\}_{0\leq t\leq T}, \tilde{\mathbb{P}},
\tilde{X}(t), \tilde{W}(t))$ such that $P^0$ is the image measure of
$\tilde{\mathbb{P}}$ by $\omega\rightarrow \tilde{X}(\cdot,
\omega)$\,.
\end{lemma}
By the uniqueness of solution, the limit of~$\nu_1^\e$, denote
by~$P^0$,  is unique and  solves the martingale
 problem related to the following stochastic partial differential equation
\begin{equation}\label{e:limit-z1}
dz_1=Az_1\,dt+ \sqrt{B(u)} \, d\overline{W} \,,
\end{equation}
where $\overline{W}(t)$ is cylindrical Wiener process with trace
operator $Q=\operatorname{Id}_H$, identity operator on~$H$, defined
on a probability space~$(\bar{\Omega}, \bar{\mathcal{F}},
\bar{\PX})$ such that $z^\e_1$~converges in probability $\bar{\PX}$
to~$z_1$ in~$C(0, T; H)$.

On the other hand, the distribution of~$z^\e_2$ on~$C(0, T; H)$ is
also tight. Suppose $z_2$~is one weak limit point of~$z^\e_2$ in~$C(0, T; H)$. We determine the equation satisfied by~$z_2$. From~(\ref{e:z2})
\begin{eqnarray*}
 \dot{z}_2^\e=Az_2^\e+\frac{1}{\sqrt{\e}}\left[f(u^\e, v^\e)-f(u^\e,
 \eta_u^\e)\right]+f'_u(\tilde{u}^\e, \eta_u^\e)z^\e
\end{eqnarray*}
for the convex combination $\tilde{u}^\e=\theta u^\e+(1-\theta)
u$\,, $\theta\in(0, 1)$.
 By assumption~$\mathbf{H}_1$ and~(\ref{e:stat})
\begin{equation*}
\frac{1}{\sqrt{\e}}\mathbb{E}|f(u^\e, v^\e)-f(u^\e,
 \eta_u^\e)|\leq
 \frac{1}{\sqrt{\e}}C_f\mathbb{E}|v^\e-\eta_u^\e|\rightarrow
 0 \text{ as } \e\rightarrow 0 \,.
\end{equation*}
And for any $h\in H^2$,
\begin{eqnarray*}
&&\mathbb{E}\langle f'_u(\tilde{u}^\e,
\eta_u^\e)z^\e-\overline{f'_u}(u) z, h\rangle\\&=&\mathbb{E}
\big\langle\big[ f'_u(\tilde{u}^\e, \eta_u^\e)-f'_u(u,
\eta_u^\e)\big]z^\e, h\big\rangle+\mathbb{E}\big\langle\big[ f'_u(u,
\eta_u^\e)-\overline{f'_u}(u)\big]z^\e, h\big\rangle
\\&&{}+ \mathbb{E}\big\langle \overline{f'_u}(u)z^\e-\overline{f'_u}(u)z,
h\big\rangle\\&\rightarrow& 0 \text{ as }  \e\rightarrow 0 \,,
\end{eqnarray*}
where $z=z_1+z_2$\,. Then $z_2$~solves the following equation
\begin{equation}\label{e:limit-z2}
\dot{z}_2=Az_2+\overline{f'_u}(u)z \,, \quad  z_2(0)=0 \,.
\end{equation}
And by the wellposedeness of the above problem, we have that~$z^\e$
uniquely converges in distribution to~$z$ which solves~(\ref{e:z}).
This proves Theorem~\ref{thm:deviation}.
\end{proof}

\section{Application to stochastic FitzHugh--Nagumo system}
\label{sec:asfns}

\begin{figure}
\centering
\begin{tabular}{c@{}cc@{}c}
\rotatebox{90}{\hspace*{11ex}$100u^\e$}&
\includegraphics[width=0.46\textwidth]{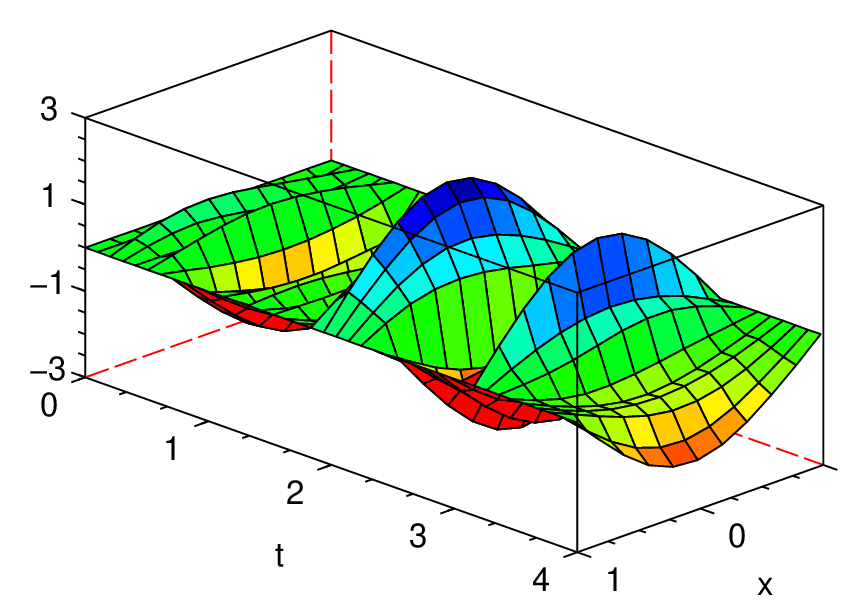}&
\rotatebox{90}{\hspace*{11ex}$10v^\e$}&
\includegraphics[width=0.46\textwidth]{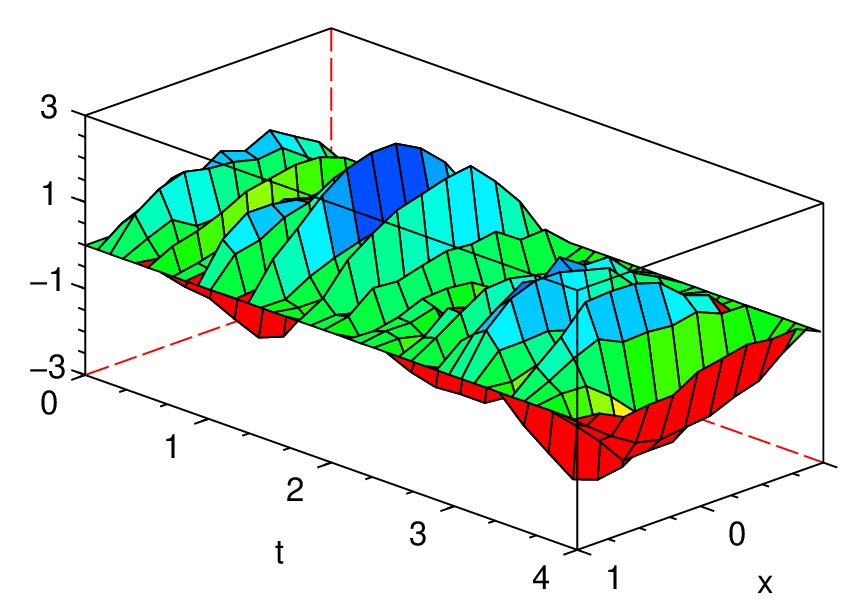}
\end{tabular}
\caption{an example realisation of the FitzHugh--Nagumo system
\eqref{e:eg1}--\eqref{e:eg2} with $\sigma_1=0$\,, $\sigma_2=3$\,,
$W=(I-A)^{-1}Z$ for a cylindrical Wiener process~$Z(t)$, small
parameter $\epsilon=0.1$\,, and domain $L=1$\,.} \label{fig:fhn}
\end{figure}

Consider the following stochastic FitzHugh--Nagumo system with
Dirichlet boundary on~$[-L,L]$:
 \begin{eqnarray}
 du^\e&=&\big[u_{xx}^\e+u^\e-(u^\e)^3+v^\e\big]\,dt\,,  \label{e:eg1}\\
 dv^\e&=&\frac{1}{\e}\big[v_{xx}^\e-v^\e+u^\e\big]\,dt+\frac{3}{\sqrt{\e}}\,dW.
 \label{e:eg2}
 \end{eqnarray}
$W(t)$ is a $L^2(-L, L)$-valued Wiener process with covariance~$Q$.
Let $A=\p_{xx}$ with zero Dirichlet boundary on $(-L, L)$\,, $f(u,
v)=u-u^3+v$\,, $g(u, v)=-v+u$\,, $\sigma_1=0$ and $\sigma_2=3$\,,
then~(\ref{e:eg1})--(\ref{e:eg2}) is in the form
of~(\ref{e:aue})--(\ref{e:ave}). Figure~\ref{fig:fhn} plots an
example solution of the FitzHugh--Nagumo system
\eqref{e:ue}--\eqref{e:ve} showing that the noise forcing of~$v$ feeds indirectly
into the dynamics of~$u$.

\begin{figure}
\centering
\begin{tabular}{c@{}c}
\rotatebox{90}{\hspace{20ex}$|\overline{u^\epsilon(0,t)|}$} &
\includegraphics{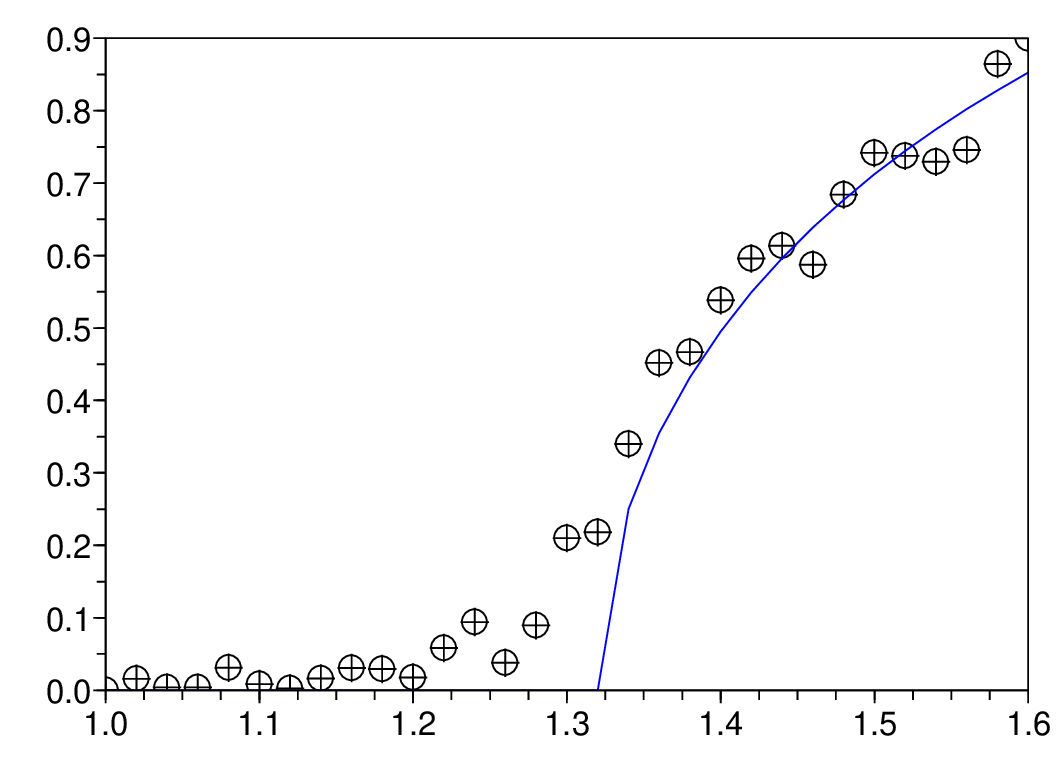}\\[-2ex]
& $L$
\end{tabular}
\caption{\textsc{rms} of the slow mode~$u^\epsilon(0,t)$ versus the half-length~$L$ at parameter $\epsilon=0.1$ for numerical simulations of~\eqref{e:eg1}--\eqref{e:eg2}.  This clearly shows the bifurcation that the averaged equation~\eqref{e:averaged} predicts from $L=\pi/2^{5/4}=1.3207$ as shown by the curve.}
\label{fig:fhnbif}
\end{figure}

Furthermore, for any fixed~$u$ the \spde~(\ref{e:eg2}) has a unique stationary
solution~$\eta_u^\e$ with distribution
 \begin{equation*}
 \mu_u=\mathcal{N}\left((I-\partial_{xx})^{-1}u,\frac{9(I-\partial_{xx})^{-1}Q}{2}+(I-\partial_{xx})^{-2}u\otimes u\right).
 \end{equation*}
Then
\begin{equation*}
\bar{f}(u)=u-u^3+(I-\partial_{xx})^{-1}u \,,
\end{equation*}
and the averaged equation is the deterministic \pde
\begin{equation}\label{e:averaged}
du=\big[\partial_{xx} u+u-u^3+(I-\partial_{xx})^{-1}u\big]\,dt\,.
\end{equation}
This averaged equation predicts a bifurcation as $L$~increases.  For a fundamental mode on~$(-L,L)$ of $u=a\cos kx$ for wavenumber $k=\pi/(2L)$, the linear dynamics of the deterministic averaged \pde~\eqref{e:averaged} predicts that $u=0$ is stable for $k>2^{1/4}$, that is, $L<\pi/2^{5/4}$.  For larger domains with $L>\pi/2^{5/4}=1.3209$ the averaged \pde\ predicts a bifurcation to finite amplitude solutions.  This bifurcation matches well with numerical solutions as seen in Figure~\ref{fig:fhnbif} which plots the mean mid-value as a function of~$L$: the bifurcation is clear albeit stochastic.

\begin{figure}
 \centering
\begin{tabular}{cc}
\rotatebox{90}{\hspace*{21ex}$100z$}&
\includegraphics{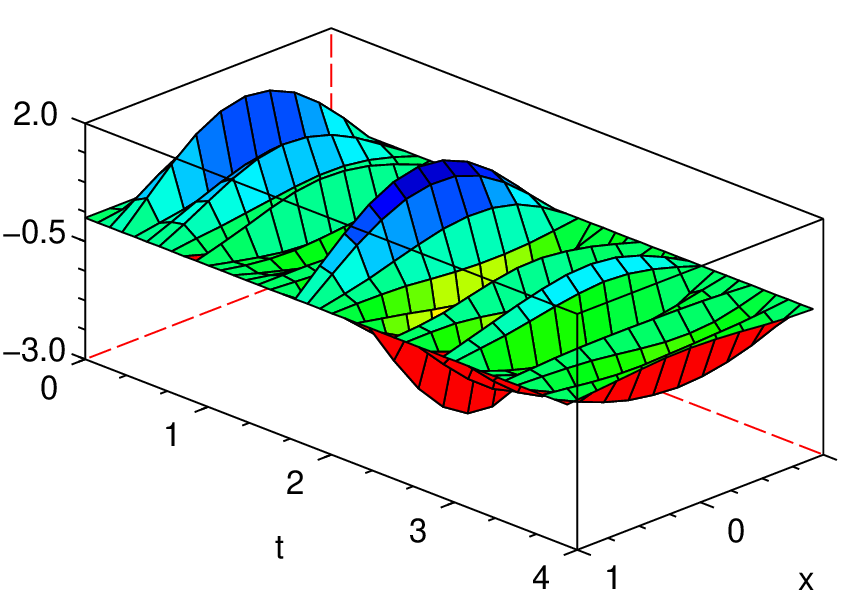}
\end{tabular}
\caption{an example realisation of the stochastic
deviation~\eqref{e:deviation} for small parameter $\epsilon=0.1$\,,
$\sigma_2=3$\,, $Q=(1-\p_{xx})^{-1}$ and $L=1$.} \label{fig:fhnav}
\end{figure}

To quantify the fluctuations evident in the dynamics of~$u^\epsilon$ we turn to the \pde\ for deviations.
Denote by~$\eta_u$ the unique stationary solution of
\begin{equation*}
dv=\big[\p_{xx}v-v+u\big]\,dt+dW.
\end{equation*}
Then $f(u, \eta_u)-\bar{f}(u)=\eta_u(t)-(I-\partial_{xx})^{-1}u$ and we have that the
deviation~$z$ solves the \spde
\begin{equation}\label{e:deviation}
dz=\big[\p_{xx}
z+(1-3u^2)z\big]\,dt+3(I-\partial_{xx})^{-1}\sqrt{Q}\,d\bar{W}
\end{equation}
with $\bar{W}(t)$ being a cylindrical Wiener process defined on a larger probability space with covariance operator~$\operatorname{Id}$ on $L^2([-L, L])$. Figure~\ref{fig:fhnav} plots an example of the deviation between the original system and the averaged system, the \spde~\eqref{e:deviation}. Including the deviation \spde~\eqref{e:deviation} gives a much better approximation than the deterministic averaged equation~\eqref{e:averaged}. In particular, when the initial state $u_0=0$ and there is no direct forcing of~$u$, $\sigma_1=0$\,, as used in Figures \ref{fig:fhn}~and~\ref{fig:fhnav}, then the averaged solution is identically $u(t)=0$\,.  In such a case, the dynamics of~$u$ as seen in Figure~\ref{fig:fhn} are modelled solely by deviations governed by the \spde~\eqref{e:deviation}.

\begin{figure}
\centering
\begin{tabular}{c@{}c}
\rotatebox{90}{\hspace{20ex}$100\,\text{variance}$} &
\includegraphics{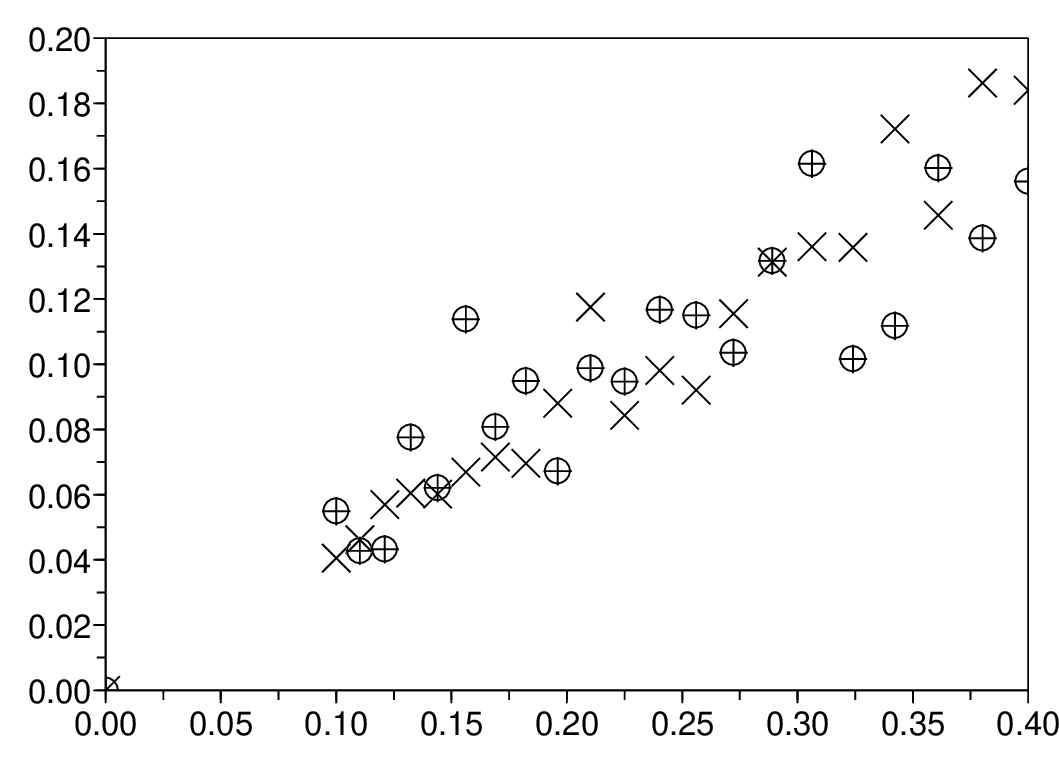}\\[-2ex]
& $\epsilon$
\end{tabular}
\caption{simulations over a time of~$128$ (with $L=1$) show:
circles, fluctuations in~$u^\epsilon(0,t)$ whose variance is plotted
as function of~$\epsilon$  ($\Delta t=0.002$); crosses, fluctuations
in~$\sqrt\epsilon z$ whose variance is plotted as a function
of~$\epsilon$ ($\Delta t=0.02$).  The two agree remarkably well:
their standard deviations are best fitted by the nearly identical
$0.065\sqrt\epsilon$ and $0.066\sqrt\epsilon$, respectively.}
\label{fig:fhnsig}
\end{figure}

To quantify the comparison between the deviation \pde~\eqref{e:deviation} and the original dynamics of the FitzHugh--Nagumo system~\eqref{e:eg1}--\eqref{e:eg2}, we look at how the fluctuations scale with scale separation parameter~$\epsilon$.  At $L=1$ the noise free state $u=u^\epsilon=0$ is stable.  As parameter~$\epsilon$ increases the fluctuations in~$u^\epsilon$ have variance as plotted by circles in Figure~\ref{fig:fhnsig}.  The scatter in the plot reflects that averages over much longer times would be better.  However, the variance does scale with~$\epsilon$ as required, and in close correspondence to that predicted by the deviation \pde~\eqref{e:deviation} (crosses).

\paragraph{Acknowledgements} This research was supported by the
Australian Research Council grant DP0774311 and NSFC grant
10701072.

\end{document}